\documentclass[a4paper, 12pt, titlepage, fleqn]{article}
\usepackage{times}
\usepackage{amsthm}
\newtheorem{thm}{Theorem}
\newtheorem{cor}{Corollary}

\newtheorem{lemma}[thm]{Lemma}
\newtheorem{prop}{Proposition}

\usepackage{amssymb,amsmath}

\DeclareMathOperator{\F}{\mathbb{F}}

\DeclareMathOperator{\Tr}{Tr}

\begin{document}
		\baselineskip=16.3pt
		\parskip=14pt
		\begin{center}
			\section*{Number of Rational points of the Generalized Hermitian Curves over $\mathbb F_{p^n}$}
			{\large 
		
			Emrah Sercan Y{\i}lmaz \footnote {Research supported by Science Foundation Ireland Grant 13/IA/1914} 
			\\
			School of Mathematics and Statistics\\
			University College Dublin\\
			Ireland}

		\end{center}
		
		\subsection*{Abstract}
In this paper we consider the curves $H_{k,t}^{(p)} : y^{p^k}+y=x^{p^{kt}+1}$ over $\mathbb F_p$ and and find an exact formula for the number of $\F_{p^n}$-rational points
on $H_{k,t}^{(p)}$ for all integers $n\ge 1$. We also give the condition when the $L$-polynomial of a Hermitian curve divides the $L$-polynomial of another over $\F_p$.

\textbf{Keywords:}  Maximal curves, Hermitian curves, Point counting, L-poynomials
\section{Introduction}
Let $q$ be a prime power and $\mathbb{F}_q$ be the finite field with $q$ elements.
Let $X$ be a projective smooth absolutely irreducible curve of genus $g$ 
defined over $\mathbb{F}_q$.
Consider the $L$-polynomial of the curve $X$ over $\mathbb F_{q}$, defined by
$$L_{X/\mathbb{F}_q}(T)=L_X(T)=\exp\left( \sum_{n=1}^\infty ( \#X(\mathbb F_{q^n}) - q^n - 1 )\frac{T^n}{n}  \right).$$
where $\#X(\mathbb F_{q^n})$ denotes the number of $\mathbb F_{q^n}$-rational points of $X$. 
It is well known that $L_X(T)$ is a polynomial of degree $2g$ with integer coefficients, so we write it as 
\begin{equation} \label{L-poly}
L_X(T)= \sum_{i=0}^{2g} c_i T^i, \ c_i \in \mathbb Z.
\end{equation}
It is also well known that $c_0=1$ and $c_{2g}=q^g$.

Let $k$ be a positive integer and $t$ be a nonnegative integer. In this paper we will study 
the curves $$H_{k,t}^{(p)} :y^{p^k}+y=x^{p^{kt}+1}$$ which are defined over $\mathbb F_p$. The curve $H_{k,t}^{(p)}$ has genus $p^{kt}(p^{k}-1)/2$ except for $(p,t)=(2,0)$, and if $(p,t)=(2,0)$, the genus of $H_{k,t}^{(p)}$ is $0$.

In this paper we will prove the following theorem which gives the number of rational points over all extension $\F_p$.

\begin{thm}\label{thm-Hkt-point}
	Let $k=2^vw$ where $w$ is an odd integer. Let $n$ be a positive integer with $d=(n,4kt)$ and $c=(n,4k)$. Then $$-p^{-n/2}[\#H_{k,t}(\mathbb F_{p^n})-(p^n+1)]=
	\begin{cases}
	0 &\text{ if }\ 2^{v+1} \nmid d \text{ and } d \mid kt,\\ 
	\epsilon\cdot(p^{c/2}-1) &\text{ if }\ 2^{v+1} \mid\mid d \text{ and } d \mid kt,\\ 
	\epsilon\cdot(p^{c/4}-1) &\text{ if }\ 2^{v+2} \mid d \text{ and } d \mid kt,\\
	p^{c}-1 &\text{ if }\ d\nmid kt, \frac d2 \mid kt \text{ and $t$ is even },\\
	-(p^{c}-1)(p^{d/2}-1) &\text{ if }\ d\nmid kt, \frac d2 \mid kt \text{ and $t$ is odd },\\
	(p^{c}-1)(p^{d/4}-1) &\text{ if }\ d\nmid 2kt, d \mid 4kt
	\end{cases}$$  
where $\epsilon=0$ if $p=2$ and $\epsilon=1$ if $p$ is odd. 
\end{thm}
When $t=1$ this curve is very well known to be a maximal curve over $\F_{p^{2k}}$. We have calculated the number of rational points over all extension of $\F_p$. We will write $H_{k}^{(p)}$ for $H_{k,1}^{(p)}$.
\begin{cor}\label{cor-Hk-point}
	Let $k=2^vw$ where $w$ is an odd integer. Let $n$ be a positive integer with $d=(n,4k)$. Then $$-p^{-n/2}[\#H_k^{(p)}(\mathbb F_{p^n})-(p^n+1)]=
	\begin{cases}
	0 &\text{ if }\ 2^{v+1} \nmid d,\\ 
	-p^{d/2}(p^{d/2}-1) &\text{ if }\ 2^{v+1} \mid\mid d,\\ 
	p^{d/4}(p^{d/4}-1) &\text{ if }\ 2^{v+2} \mid d.
	\end{cases}$$  
\end{cor}
We will usually drop the superscript in $H_{k,t}^{(p)}$ and write $H_{k,t}$.
Throughout the paper $\left(\frac{\cdot }{p}\right)$ denotes the Legendre symbol.

In Section 2, we will give some backgrounds which we will use in our proofs. In Section 3, we will give the proofs of two general tools which are useful for proving our results. In Section 4, we will find the $F_{p^n}$-number of rational points of $H_{k,0}$ where $p$ is an odd prime. In Section 5, we will find the $F_{p^n}$-number of rational points of $H_{k,t}$ where $p$ is an odd prime.  In Section 6, we will find the $F_{p^n}$-number of rational points of $H_{k,t}$ where $p=2$.  In Section 7 and 8, we will give some remarks on maximal curves which are related to our previous results. In Section 9, we will prove that in which condition the $L$-polynomial of a Hermitian curve divides the $L$-polynomial of another Hermitian curve over $\mathbb F_p$.
\section{Background}
In this section we will give some basic facts that we will use. 
Some of this requires that $p$ is odd, some of it does not.

\subsection{More on Curves}\label{morec}

Let $X$ be a projective smooth absolutely irreducible curve of genus $g$ 
defined over $\mathbb{F}_q$. Let $\eta_1,\cdots,\eta_{2g}$ be the roots of the 
reciprocal of the $L$-polynomial of $X$ over 
$\mathbb F_{q}$ (sometimes called the Weil numbers of $X$, or
Frobenius eigenvalues). Then, for any $n\geq 1$, 
the number of rational points of $X$ over $\mathbb F_{q^{n}}$ is
given by
\begin{equation}\label{eqn-sum of roots}
\#X(\mathbb F_{q^{n}})=(q^{n}+1)- \sum\limits_{i=1}^{2g}\eta_i^n. 
\end{equation}
The Riemann Hypothesis for curves over finite fields  
states that $|\eta_i|=\sqrt{q}$ for all $i=1,\ldots,2g$. 
It follows immediately from this property and \eqref{eqn-sum of roots}
that
\begin{equation}
|\#X(\mathbb F_{q^n})-(q^n+1)|\leq 2g\sqrt{q^n}
\end{equation}
which is the Hasse-Weil bound.

We call $X(\mathbb F_{q})$ \emph{maximal} if $\eta_i=-\sqrt{q}$ for all $i=1,\cdots,2g$, so the Hasse-Weil upper bound is met. Equivalently, $X(\mathbb F_{q})$ is maximal
if and only if $L_X(T)=(1+\sqrt{q} T)^{2g}$.

We call $X(\mathbb F_{q})$ \emph{minimal} if $\eta_i=\sqrt{q}$ for all $i=1,\cdots,2g$,
so the Hasse-Weil lower bound is met.
Equivalently, $X(\mathbb F_{q})$ is minimal
if and only if $L_X(T)=(1-\sqrt{q} T)^{2g}$.

Note that if $X(\mathbb F_{q})$ is minimal or maximal then $q$ must be a square 
(i.e.\ $r$ must be even).

The following properties follow immediately.

\begin{prop} \label{minimal-prop}
	\begin{enumerate}
		\item If $X(\mathbb F_{q})$ is maximal then  $X(\mathbb F_{q^{n}})$ is minimal for even $n$ and maximal for odd $n$. 
		\item  If $X(\mathbb F_{q})$ is minimal then  $X(\mathbb F_{q^{n}})$ is minimal for all $n$.
	\end{enumerate}
\end{prop}

\begin{prop}\label{pureimag} \cite{MY} 
If $X$ is a curve defined over $\F_q$ and $X(\mathbb F_{q^{2n}})$ is maximal,
then $\#X(\mathbb F_{q^n})=q^n+1$.
\end{prop}
The best known example of a maximal curve is $H_k$ over $\mathbb F_{p^{2k}}$.

\subsection{Supersingular Curves}\label{supsing}

A curve $X$ of genus $g$ defined over $\mathbb F_q$ 
($q=p^r$) is \emph{supersingular} if any of the following 
equivalent properties hold.

\begin{enumerate}
	\item All Weil numbers of $X$ have the form $\eta_i = \sqrt{q}\cdot \zeta_i$ where $\zeta_i$ is a root of unity.
	\item The Newton polygon of $X$ is a straight line of slope $1/2$.
	\item The Jacobian of $X$ is geometrically isogenous to $E^g$ where
	$E$ is a supersingular elliptic curve.
	\item If $X$ has $L$-polynomial
	$L_X(T)=1+\sum\limits_{i=1}^{2g} c_iT^i$ 
	then
	$$ord_p(c_i)\geq \frac{ir}{2}, \ \mbox{for all $i=1,\ldots
		,2g$.}$$
\end{enumerate}

By the first property, a supersingular curve defined over $\mathbb F_q$ becomes minimal over some finite extension of $\mathbb F_q$.
Conversely, any minimal or maximal curve is supersingular.

\subsection{Quadratic forms}\label{QF}

We now recall the basic theory of quadratic forms over $\mathbb{F}_{q}$, where $q$ is odd.

Let $K=\mathbb{F}_{q^n}$, and 
let $Q:K\longrightarrow \mathbb{F}_{q}$ be a quadratic form.
The polarization of $Q$ is the symplectic bilinear form $B$ defined by $B(x,y)=Q(x+y)-Q(x)-Q(y)$. By definition the radical of $B$ (denoted $W$) is $
W =\{ x\in K : B(x,y)=0 \text{  for all $y\in K$}\}$. The rank of $B$ is defined to be $n-\dim(W)$. The rank of $Q$ is defined to be the rank of $B$.

The following result is well known, see Chapter $6$ of \cite{lidl}  for example.
\bigskip

\begin{prop}\label{counts}
	Continue the above notation. Let $N=|\{x\in K : Q(x)=0\}|$, and let $w=\dim(W)$. If $Q$ has odd rank then $N=q^{n-1}$; if $Q$ has even rank then $N=q^{n-1}\pm (q-1)q^{(n-2+w)/2}$.
\end{prop}

In this paper we will be concerned with quadratic forms of the type
$Q(x)=\Tr(f(x))$ where $f(x)$ has the form $\sum a_{ij} x^{q^i+q^j}$ for any prime power $q$ and $\Tr$ maps to $\mathbb F_q$.
If $N$ is the number of $x\in \mathbb{F}_{q^n}$ with $\Tr(f(x))=0$, then
because elements of trace 0 have the form $y^q-y$,  finding $N$ is equivalent
to finding the exact number of $\mathbb{F}_{q^n}$-rational points on 
the curve $C: y^q-y=f(x)$.
Indeed, 
\begin{equation}\label{quadpts}
\#C(\mathbb F_{q^n})=qN+1.
\end{equation}

\subsection{Relations on the Number of Rational Points}

In this section we state a theorem which allows us to find the number of 
$\F_{p^n}$-rational points of a supersingular curve by finding the the number of 
$\F_{p^m}$-rational points only for the divisors $m$ of $s$, where the Weil numbers 
are $\sqrt{p}$ times an $s$-th root of unity. Note that $s$ is even because
equality holds in the Hasse--Weil bound over $\F_{p^s}$.

\begin{thm}[\cite{MY2}]\label{reduction-thm} 	
	Let $X$ be a supersingular curve of genus $g$ defined over $\mathbb F_q$ with
period $s$.
Let $n$ be a positive integer, let $\gcd (n,s)=m$ and write $n=m\cdot t$. If $q$ is odd, then we have $$
\#X(\F_{q^n})-(q^n+1)=\begin{cases}
q^{(n-m)/2}[\#X(\F_{q^m})-(q^m+1)] &\text{if } m\cdot r \text{ is even},\\
q^{(n-m)/2}[\#X(\F_{q^m})-(q^m+1)]&\text{if } m \cdot r \text{ is odd and } p\mid t,
\\
q^{(n-m)/2}[\#X(\F_{q^m})-(q^m+1)]\left(\frac{(-1)^{(t-1)/2}t}{p}\right)&\text{if } m \cdot r \text{ is odd and } p\nmid t.
\end{cases}$$ If $q$ is even, then we have  $$
\#X(\F_{q^n})-(q^n+1)=\begin{cases}
q^{(n-m)/2}[\#X(\F_{q^m})-(q^m+1)] &\text{if } m\cdot r \text{ is even},\\
q^{(n-m)/2}[\#X(\F_{q^m})-(q^m+1)](-1)^{(t^2-1)/8}&\text{if } m \cdot r \text{ is odd}.\\
\end{cases}$$
\end{thm}

\subsection{Discrete Fourier Transform}

In this section we recall the statement of  the Discrete Fourier Transform
and its inverse.

\begin{prop}[Inverse Discrete Fourier Transform]
	Let $N$ be a positive integer and let $w_N$ be a primitive $N$-th root of unity
	over any field where $N$ is invertible.  If $$F_n=\sum_{j=0}^{N-1}f_jw_N^{-jn}$$ for $n=0,1\cdots, N-1$ then we have  $$f_n=\frac1{N}\sum_{j=0}^{N-1}F_jw_N^{jn}$$ for $n=0,1\cdots, N-1$.
\end{prop}

\subsection{Divisibilty Theorems}

The following theorem is well-known.

\begin{thm}\label{chapman:KleimanSerre} 
(Kleiman--Serre)
If there is a surjective morphism of curves 
$C \longrightarrow D$ that is defined over $\mathbb{F}_q$
then $\mathrm{L}_{D}(T)$ divides $\mathrm{L}_C(T)$.
\end{thm}

Moreover, If there is a surjective morphism of curves 
$C \longrightarrow D$ that is defined over $\mathbb{F}_q$, we will say that $D$ is covered by $C$.

However, there are
cases where there is no map of curves and yet there is divisibility of
 L-polynomials. 
We suspect that $H_k^{(p)} $ and $H_{2k}^{(p)} $ is such a case, see Corollary \ref{divisibility-2k}.
%We are unable to find a map $H_{2k}^{(p)} \longrightarrow H_{k}^{(p)}$ when $p>2$.

\begin{prop}[\cite{MY}]\label{divisiblity-period}
	Let $C$ and $D$ be supersingular curves over $\mathbb F_q$. If $L(C)$ divides $L(D)$, then $s_C$ divides $s_D$.
\end{prop}

\section{The Image of the Maps $y\to y^{p^k}\pm y$ over $\F_{p^n}$}
In this section, we will give some useful tools. We will use these tools in sections \ref{sectHk} and \ref{sectHk-even}.
\begin{lemma}\label{prop-y^p-y}
	Let $n$ and $k$ be integers with $(n,k)=d$. For integer $m\ge 1$ define $$S_m=\{y^{p^m}-y \ : \ y \in \mathbb F_{p^n}\}.$$ Then the set equality $S_k=S_d$ holds.
\end{lemma}

\begin{proof}
Since both $y\to y^{p^k}-y$ and $y\to y^{p^d}-y$ are additive group homomorphism from $\mathbb F_{p^n}$ to $\mathbb F_{p^n}$ with kernel $\mathbb F_{p^d}$, the cardinalities of $S_k$ and $S_d$ equal. Since $$y^{p^k}-y=\left(\sum_{i=0}^{k/d-1}y^{p^{di}}\right)^{p^d}-\sum_{i=0}^{k/d-1}y^{p^{di}},$$ we have $S_k\subseteq S_d$ and therefore $S_k=S_d$. 
	
\end{proof}

Note that if $p=2$, we will prefer to use Lemma \ref{prop-y^p-y} instead of the following lemma. 

\begin{lemma}\label{prop-y^p+y}
	Let $n$ and $k$ be integers with $(n,k)=d$. Let $k/d$ is odd and $n/d$ is even. Then $$\{y^{p^d}+y \ : \ y \in \mathbb F_{p^n}\}=\{y^{p^k}+y \ : \ y \in \mathbb F_{p^n}\}.$$
\end{lemma}

\begin{proof}
Since $n/d$ is even, there exist $\mu \in \mathbb F_{p^n}^\times$ such that $\mu^{p^d}=-\mu$. Since $k/d$ is odd, we also have $\mu^{p^k}=-\mu$.

We have $$\{y^{p^d}-y \ : \ y \in \mathbb F_{p^n}\}=\{y^{p^k}-y \ : \ y \in \mathbb F_{p^n}\}$$ by Lemma \ref{prop-y^p-y}. If we multiply the sets by $-\mu^{-1}$, $$\{-\mu^{-1}(y^{p^d}-y) \ : \ y \in \mathbb F_{p^n}\}=\{-\mu^{-1}(y^{p^k}-y) \ : \ y \in \mathbb F_{p^n}\}$$ holds. Since $\{\mu y \ : \  y \in \mathbb F_{p^n}\}=\{y \ : \  y \in \mathbb F_{p^n}\}$,   $$\{-\mu^{-1}((\mu y)^{p^d}-(\mu y)) \ : \  y \in \mathbb F_{p^n}\}=\{-\mu^{-1}((\mu y)^{p^k}-(\mu y)) \ : \  y \in \mathbb F_{p^n}\}$$ holds. Simply put,	$$\{y^{p^d}+y \ : \ y \in \mathbb F_{p^n}\}=\{y^{p^k}+y \ : \ y \in \mathbb F_{p^n}\}$$ holds.	

\end{proof}

\section{The Number of $\mathbb F_{p^n}$-Rational Points of $y^{p^k}+y=x^2$ where $p$ is odd}\label{sectHk0}
Let $p$ be an odd prime. In this section we will find the number of rational points of $H_{k,0}$ over $\F_{p^n}$ for all $n\ge 1$ where $k$ is a positive integer and prove Theorem \ref{thm-Hk0-point}.
\begin{lemma}\label{lemma-Hk-max-min-x^2}
	Then the curve $H_{k,0}$ is maximal over $\F_{p^{2k}}$ and minimal over $\F_{p^{4k}}$.
\end{lemma}
\begin{proof}	
	Let $\mu$ be a nonzero element in $\mathbb F_{p^{2k}}$ such that $\mu^{p^k}=-\mu$. Since $(x,y) \mapsto(-\mu x,-\mu y)$ is an one-to-one map from $\mathbb F_{p^{2k}}$ to $\mathbb F_{p^{2k}}$, we have $\#H_k(\F_{p^{2k}})$ equals the number of $\F_{p^{2k}}$-rational points of $$G_{k,0,\mu}: y^{p^k}-y=\mu x^{2}.$$ We have $$
	\Tr_{\F_{p^{2k}}/\F_{p^k}}(\mu x^{2})=\mu x^{2}+\mu^{p^{k}}x^{2p^k}=\mu (x^2-x^{2p^k})
	$$
	for all $x\in \F_{p^{2k}}$. 	Since $\deg (x^{2(p^k-1)}-1,x^{p^{2k}-1}-1)=2(p^k-1)$, we have that the number of $\mathbb F_{p^{2kt}}$-rational points of $G_{k,d,\mu}$ is $$1+p^k(1+2(p^k-1))=p^{2k}+1+(p^k-1)p^k.$$
	by (\ref{quadpts}). Therefore, we have $$H_{k,t}(\F_{p^{2kt}})=G_{k,t,\mu}(\F_{p^{2kt}})=(p^{2k}+1)-(p^k-1)p^{k}.$$
	Since the genus of the curve $H_{k,0}$ is $(p^k-1)/2$, $H_{k,0}$ is maximal over $\F_{p^{2k}}$ and therefore it is minimal over $\F_{p^{4k}}$ by Proposition \ref{minimal-prop}.
\end{proof}
\begin{lemma}\label{lem-divides-k-0}
	Let $d\mid k$. The number of $\mathbb F_{p^d}$-rational points of $H_{k,0}$ is $p^d+1$.
\end{lemma}
\begin{proof}
	Since $u^{p^{k}}=u$ for each $u\in \mathbb F_{p^d}$, we have $$0=y^{p^k}+y-x^{p^{k}+1}=2y-x^2.$$
	Since for each $x\in \mathbb F_{p^d}$ there exists a unique $y\in \mathbb F_{p^d}$ satisfying the above equality, we have the result.
\end{proof}

\begin{lemma}\label{lem-divides-2k-0}
	Let $d\mid k$ with $2d \nmid k$. The number of $\mathbb F_{p^{2d}}$-rational points of $H_{k,0}$ is $p^{2d}+1+(p^d-1)p^d$.
\end{lemma}
\begin{proof}
	Write $k=2^ut$ where $u$ is a non-negative integer and $t$ is odd. There exists an odd $e$ such that $k=d\cdot e$ and $e \mid t$. Since $e$ is odd, there exists an integer $f$ such that $e=2f+1$.
	
	Let $u \in \mathbb F_{p^{2d}}$.  Then $$u^{p^k}=u^{p^{d\cdot e}}=u^{p^{d\cdot(2f+1)}}=u^{p^{2d\cdot f+ d}}=u^{p^d}.$$ Hence we have $$y^{p^k}+y-x^{2}=y^{p^d}+y-x^{2}$$ for all $x,y \in \mathbb F_{p^{2d}}$. Therefore, we have $$\#H_{k,0}(\mathbb F_{p^{2d}})=\#H_{d,0}(\F_{p^{2d}}).$$ Since $H_{d,0}$ is maximal over $\mathbb F_{p^{2d}}$ by Theorem \ref{lemma-Hk-max-min-x^2}, we have the result. 
\end{proof}

\begin{lemma}\label{lem-divides-4k-x^2}
	Let $d\mid 4k$ with $d \nmid 2k$. The number of $\mathbb F_{p^{d}}$-rational points of $H_{k,0}$ is $p^{d}+1+(p^{d/4}-1)p^{d/2}$.
\end{lemma}
\begin{proof}
	Write $d=4e$ where $e$ is an integer and $k=2^ut$ where $u$ is a non-negative integer and $t$ is odd. There exists an odd $f$ such that $k=e\cdot f$ and $f \mid t$. Since $f$ is odd, there exists an integer $g$ such that $e=4g+1$ or $e=4g-1$.
	
	\textbf{Case A. $\mathbf{e=4g+1.}$} We have $$u^{p^k}=u^{p^{e\cdot f}}=u^{p^{e\cdot(4g+1)}}=u^{p^{4e\cdot g+ e}}=u^{p^{e}}$$ for all $u \in \mathbb F_{p^{4e}}$. Hence we have $$y^{p^k}+y-x^{p^k+1}=y^{p^e}+y-x^{2}$$ for all $x,y \in \mathbb F_{p^{4e}}$. Therefore, we have $$\#H_{k,0}(\mathbb F_{p^{4e}})=\#H_{e,0}(\F_{p^{4e}}).$$ Since $H_{e,0}$ is minimal over $\mathbb F_{p^{4e}}$ by Lemma \ref{lemma-Hk-max-min-x^2} , we have the result.

	\textbf{Case B. $\mathbf{e=4g-1.}$} We have	$$u^{p^k}=u^{p^{e\cdot f}}=u^{p^{e\cdot(4g-1)}}=u^{p^{4e\cdot (g-1)+3e}}=u^{p^{3e}}$$ for all $u \in \mathbb F_{p^{4e}}$. Hence we have $$y^{p^k}+y-x^{p^k+1}=y^{p^{3e}}+y-x^{p^{3e}+1}=(y^{p^e}+y-x^{p^e+1})^{p^{3e}}$$ for all $x,y \in \mathbb F_{p^{4e}}$. Therefore, we have $$\#H_{k,0}(\mathbb F_{p^{4e}})=\#H_{e,0}(\F_{p^{4e}}).$$ Since $H_{e,0}$ is minimal over $\mathbb F_{p^{d}}$ by Lemma \ref{lemma-Hk-max-min-x^2} , we have the result. 
\end{proof}

\begin{thm}\label{thm-Hk0-point}
	Let $k=2^vt$ where $t$ is an odd integer. Let $n$ be a positive integer with $d=(n,4k)$. Then $$-p^{-n/2}[\#H_{k,0}(\mathbb F_{p^n})-(p^n+1)]=
	\begin{cases}
	0 &\text{ if }\ 2^{v+1} \nmid d,\\ 
	-(p^{d/2}-1)&\text{ if }\ 2^{v+1} \mid\mid d,\\ 
p^{d/4}-1 &\text{ if }\ 2^{v+2} \mid d.
	\end{cases}$$  
\end{thm}
\begin{proof}
It follows by Lemma \ref{lem-divides-k-0}, \ref{lem-divides-2k-0} and \ref{lem-divides-4k-x^2}.
\end{proof}
\section{The Number of $\mathbb F_{p^n}$-Rational Points of $H_{k,t}$ where $p$ is odd}\label{sectHk}

Let $p$ be an odd prime. In this section we will find the number of rational points of $H_{k,t}$ over $\F_{p^n}$ for all $n\ge 1$ where $k$ and $t$ are positive integers and mainly prove the Theorem \ref{thm-Hkt-point} for odd primes.
\begin{lemma}\label{lemma-Hk-max-min-odd}
	Let $t$ be an odd integer. Then the curve $H_{k,t}$ is maximal over $\F_{p^{2kt}}$ and minimal over $\F_{p^{4kt}}$.
\end{lemma}
\begin{proof}
	Let $\mu$ be a nonzero element in $\mathbb F_{p^{2k}}$ such that $\mu^{p^k}=-\mu.$ Since $(x,y)\mapsto(\mu x,\mu y)$ is an one-to-one map from $\mathbb F_{p^{2kt}}$ to $\mathbb F_{p^{2kt}}$, we have $\#H_k(\F_{p^{2kt}})$ equals to number of $\F_{p^{2kt}}$-rational points of $$G_{k,t,\mu}: y^{p^k}-y=\mu x^{p^{kt}+1}.$$ Since \begin{align*}
	\Tr_{\F_{p^{2kt}}/\F_{p^k}}(\mu x^{p^{kt}+1})
	&=\Tr_{\F_{p^{kt}}/\F_{p^k}}[\Tr_{\F_{p^{2kt}}/\F_{p^{kt}}}(\mu x^{p^{kt}+1}+\mu^{p^{kt}}x^{p^{2kt}+p^{kt}})]\\
	&=\Tr_{\F_{p^{kt}}/\F_{p^k}}(x^{p^{kt}+1}(\mu+\mu^{p^{kt}}))\\
	&=\Tr_{\F_{p^{kt}}/\F_{p^k}}(0)\\
	&=0
	\end{align*}
	for all $x\in \F_{p^{2k}}$, by (\ref{quadpts}) we have $$H_{k,t}(\F_{p^{2kt}})=G_{k,t,\mu}(\F_{p^{2kt}})=1+p^{2kt}\cdot p^{k}=(p^{2kt}+1)+p^{kt}(p^k-1)\sqrt{p^{2kt}}.$$ Since the genus of the curve $H_{k,t}$ is $p^{kt}(p^k-1)/2$, $H_{k,t}$ is maximal over $\F_{p^{2kt}}$ and therefore it is minimal over $\F_{p^{4kt}}$ by Proposition \ref{minimal-prop}.
\end{proof}

\begin{lemma}\label{lemma-Hk-max-min-even}
	Let $t$ be an even integer. Then the number of rational points of $H_{k,t}$ over $\F_{p^{2kt}}$ is $p^{2kt}+1-(p^k-1)p^{kt}$ and $H_{k,t}$ is minimal over $\F_{p^{4kt}}$.
\end{lemma}
\begin{proof}
	Let $\mu$ be a nonzero element in $\mathbb F_{p^{2k}}$ such that $\mu^{p^k}=-\mu$ as in Lemma \ref{lemma-Hk-max-min-odd}. Since $(-\mu x,-\mu y)$ is an one-to-one map from $\mathbb F_{p^{2kt}}$ to $\mathbb F_{p^{2kt}}$, we have $\#H_k(\F_{p^{2kt}})$ equals to number of $\F_{p^{2kt}}$-rational points of $$G_{k,t,\mu}: y^{p^k}-y=\mu x^{p^{kt}+1}.$$ We have \begin{align*}
	\Tr_{\F_{p^{2kt}}/\F_{p^k}}(\mu x^{p^{kt}+1})
	&=\Tr_{\F_{p^{kt}}/\F_{p^k}}[\Tr_{\F_{p^{2kt}}/\F_{p^{kt}}}(\mu x^{p^{kt}+1}+\mu^{p^{kt}}x^{p^{2kt}+p^{kt}})]\\
	&=\Tr_{\F_{p^{kt}}/\F_{p^k}}(x^{p^{kt}+1}(\mu+\mu^{p^{kt}}))\\
	&=\Tr_{\F_{p^{kt}}/\F_{p^k}}(2\mu x^{p^{kt}+1})
	\end{align*}
	for all $x\in \F_{p^{2k}}$. 	Since $x\to x^{p^{kt}+1}$ is $p^{kd}+1$-to-$1$ map from $\mathbb F_{p^{2kt}}^\times$ to $\mathbb F_{p^{kt}}^\times$ and since $\Tr_{\F_{p^{kt}}/\F_{p^k}}(x)$ is a linear map from $\mathbb F_{p^{2kt}}$ to $\mathbb F_{p^k}$, we have that the number of $\mathbb F_{p^{2kt}}$-rational points of $G_{k,d,\mu}$ is $$1+p^k(1+(p^{kt}+1)(p^{k-1}-1))=(p^{2kt}+1)-(p^k-1)p^{kt}.$$
	by (\ref{quadpts}). Therefore, we have $$H_{k,t}(\F_{p^{2kt}})=G_{k,t,\mu}(\F_{p^{2kt}})=(p^{2kt}+1)-(p^k-1)p^{kt}.$$
	
	Let $Q(x)=\Tr_{\F_{p^{4kt}}/\F_{p^{k}}}(\mu x^{p^{kt}+1})$. Then \begin{align*}
	B(x,y)
	&=Q(x+y)-Q(x)-Q(y)\\
	&=\Tr_{\F_{p^{4kt}}/\F_{p^{k}}}(\mu(xy^{p^{kt}}+yx^{p^{kt}}))\\
	&=\Tr_{\F_{p^{4kt}}/\F_{p^{k}}}(\mu y^{p^{kt}}(x+x^{p^{2kt}}))
	\end{align*} is a bilinear form over $\mathbb F_{p^k}$.
	Then the radical of $B$ is $$W=\{x\in \F_{p^{4kt}} \ | \ x^{p^{2kt}}+x=0 \}.$$ Since $x^{p^{2kt}}+x$ divides $x^{p^{4kt}}-x$, we have $\dim_{\F_{p^k}} W=2t$. Since $4t-\dim_{\F_{p^k}} W=2t$ is even, by Proposition \ref{counts}, we have $$N=p^{4kt-k}\pm p^{k(4t-2+2t)/2}=p^{n(k-1)}\pm (p^k-1)p^{k(3t-1)}.$$ Hence by (\ref{quadpts}) we have $$\#H_{k,t}(\F_{p^{4kt}})=p^{4kt}+1\pm (p^k-1)p^{3kt}=p^{4kt}+1\pm (p^k-1)p^{kt}\sqrt{p^{4kt}}.$$
	Therefore, $H_{k,t}$ is either maximal or minimal over $\F_{p^{4kt}}$. It cannot be maximal by Proposition \ref{pureimag}. Therefore, it is minimal over $\F_{p^{4kt}}$.
\end{proof}

\begin{lemma}\label{lem-divides-k}
	Let $d\mid kt$. The number of $\mathbb F_{p^d}$-rational points of $H_{k,t}$ equals to the number of $\mathbb F_{p^d}$-rational points of $H_{k,0}$.
\end{lemma}
\begin{proof}
	Since $u^{p^{kt}}=u$ for each $u\in \mathbb F_{p^d}$, we have $$0=y^{p^k}+y-x^{p^{kt}+1}=y^{p^k}+y-x^2.$$
	Therefore, we have the result.
\end{proof}

\begin{lemma}\label{lem-divides-2k}
	Let $d\mid kt$ with $2d \nmid kt$ and let $(k,d)=c$. The number of $\mathbb F_{p^{2d}}$-rational points of $H_{k,t}$ equals to he number of $\mathbb F_{p^{2d}}$-rational points of $H_{c,d/c}$.
\end{lemma}
\begin{proof}
Write $kt=2^uv$ where $u$ is a non-negative integer and $v$ is odd. There exists and odd $e$ such that $kt=d\cdot e$ and $e \mid v$. Since $e$ is odd, there exists an integer $f$ such that $e=2f+1$.

Let $u \in \mathbb F_{p^{2d}}$.  Then $$u^{p^{kt}}=u^{p^{d\cdot e}}=u^{p^{d\cdot(2f+1)}}=u^{p^{2d\cdot f+ d}}=u^{p^d}.$$ Hence we have $$y^{p^k}+y-x^{p^{kt}+1}=y^{p^k}+y-x^{p^d+1}$$ for all $x,y \in \mathbb F_{p^{2d}}$. Since $(2d,k)=(d,k)$ and since $2d/c$ is even and $k/c$ is odd, by Proposition \ref{prop-y^p+y}, there exists $z\in \mathbb F_{p^{2kd}}$ such that  $$y^{p^k}+y-x^{p^{d}+1}=z^{p^c}+z-x^{p^d+1}$$ for all $x,y \in \mathbb F_{p^{2d}}$. Therefore, we have $\#H_{k,2d}(\mathbb F_{p^{2d}})=\#H_{c,d/c}(\F_{p^{2d}})$.
\end{proof}

\begin{lemma}\label{lem-divides-4k}
	Let $d\mid 4kt$ with $d \nmid 2kt$ and let $(k,d)=c$. The number of $\mathbb F_{p^{d}}$-rational points of $H_{k,t}$ equals to he number of $\mathbb F_{p^{d}}$-rational points of $H_{c,d/4c}$.
\end{lemma}
\begin{proof}
	Write $d=4e$ where $e$ is an integer and $k=2^uv$ where $u$ is a non-negative integer and $v$ is odd. There exists and odd $f$ such that $kt=e\cdot f$ and $f \mid t$. Since $f$ is odd, there exists an integer $g$ such that $e=4g+1$ or $e=4g-1$.
	
	\textbf{Case A. $\mathbf{e=4g+1.}$} We have $$u^{p^{kt}}=u^{p^{e\cdot f}}=u^{p^{e\cdot(4g+1)}}=u^{p^{4e\cdot g+ e}}=u^{p^{e}}$$ for all $u \in \mathbb F_{p^{4e}}$. Hence we have $$y^{p^k}+y-x^{p^{kt}+1}=y^{p^k}+y-x^{p^e+1}$$ for all $x,y \in \mathbb F_{p^{4e}}$. Since $d/c$ is even and $k/c$ is odd, by Proposition \ref{prop-y^p+y} we have  $$y^{p^k}+y-x^{p^{kt}+1}=y^{p^c}+y-x^{p^e+1}$$ for all $x,y \in \mathbb F_{p^{4e}}$. Therefore, we have $\#H_{k,t}(\mathbb F_{p^{d}})=\#H_{c,d/4c}(\F_{p^{d}})$.

		\textbf{Case B. $\mathbf{e=4g-1.}$} We have	$$u^{p^{kt}}=u^{p^{e\cdot f}}=u^{p^{e\cdot(4g-1)}}=u^{p^{4e\cdot (g-1)+3e}}=u^{p^{3e}}$$ for all $u \in \mathbb F_{p^{4e}}$. Hence we have $$y^{p^k}+y-x^{p^{kt}+1}=y^{p^k}+y-x^{p^{3e}+1}$$ for all $x,y \in \mathbb F_{p^{4e}}$.  Since $d/c$ is even and $k/c$ is odd, by Proposition \ref{prop-y^p+y} we have  $$y^{p^k}+y-x^{p^{kt}+1}=y^{p^c}+y-x^{p^{3e}+1}$$ for all $x,y \in \mathbb F_{p^{4e}}$.  Since $(x^{p^{3e}+1})^{p^e}=x^{p^{e}+1}$ for all $x\in F_{p^{4e}}$ and the map $(x,y) \to (x^{p^e},y)$ from $F_{p^{4e}}$ to $F_{p^{4e}}$ is one-to-one and onto, we have $\#H_{k,t}(\mathbb F_{p^{d}})=\#H_{c,4d/c}(\F_{p^{d}})$. 
\end{proof}

Now we can prove Thorem \ref{thm-Hkt-point} for odd primes.

\begin{proof}[Proof of Theorem \ref{thm-Hkt-point} where $p$ is odd]
	It follows by Lemma \ref{lem-divides-k}, \ref{lem-divides-2k}, \ref{lem-divides-4k} and Theorem \ref{reduction-thm}.
\end{proof}

\section{The Number of $\mathbb F_{p^n}$-Rational Points of $H_{k,t}$ where $p=2$}\label{sectHk-even}
This section is similar to the previous section. We  will clarify what the differences are and prove Theorem \ref{thm-Hkt-point} for $p=2$.

\begin{lemma}\label{lem-p=2-x^2}
	Let $k$ and $n$ be integers. The cardinality of the set $\{(x,y) \in \F_{2^n}^2 \: : \:   y^{p^k}+y=x^2\}$ is $p^n$.
\end{lemma}
\begin{proof}
	Since the map $x\mapsto x^2$ is an automorphism on $\mathbb F_q$ where $q$ is even, we have $$|\{(x,y) \in \F_{2^n}^2 \: : \:   y^{p^k}+y=x^2\}|=|\{(x,y) \in \F_{2^n}^2 \: : \:   y^{p^k}+y=x\}|.$$ Moreover, by Proposition \ref{prop-y^p-y}, we have $$|\{(x,y) \in \F_{2^n}^2 \: : \:   y^{p^k}+y=x\}|=|\{(x,y) \in \F_{2^n}^2 \: : \:  y^{p^d}+y=x\}|$$ where $d=(n,k)$. This cardinality basically equals to $$p^d \cdot |\{x \in \F_{2^n} \: : \:  \Tr_{\F_{2^{n}}/\F_{2^d}}(x)=0\}|=p^k\cdot p^{n-k}=p^n.$$  This finishes the proof.
\end{proof}

Note that this result in Lemma \ref{lem-p=2-x^2} is also follows by the fact that the genus of $H_{k,0}$ is $0$.
\begin{lemma}
	Let $k$ and $t$ be integers. Then $H_{k,t}$ is maximal over $\F_{2^{2kt}}$ and minimal over $\F_{2^{4kt}}$.
\end{lemma}

\begin{proof}
We want to note that every root of $x^{2^k}+x$ is in $\F_{2^k}$ and $2x=0$ for all such $x$ since we are in characteristic $2$. Now the proof is similar to the proof of Theorem \ref{lemma-Hk-max-min-odd}.
\end{proof}

\begin{lemma}\label{lem-all-p=2} The followings hold.\begin{itemize}
	\item[1.] 		Let $d\mid kt$. The number of $\mathbb F_{2^d}$-rational points of $H_{k,t}$ equals to $H_{k,0}$.
		
	\item[2.]	Let $d\mid kt$ with $2d \nmid kt$ and let $(k,d)=c$. The number of $\mathbb F_{2^{2d}}$-rational points of $H_{k,t}$ equals to he number of $\mathbb F_{2^{2d}}$-rational points of $H_{c,d/c}$.
			
	\item[3.]	Let $d\mid 4kt$ with $d \nmid 2kt$ and let $(k,d)=c$. The number of $\mathbb F_{2^{d}}$-rational points of $H_{k,t}$ equals to he number of $\mathbb F_{2^{d}}$-rational points of $H_{c,d/4c}$.
\end{itemize}
\end{lemma}

\begin{proof}
The proofs are respectively same as the proofs of Lemmas \ref{lem-divides-k}, \ref{lem-divides-2k} and \ref{lem-divides-4k}. 
\end{proof}

Now we can prove Thorem \ref{thm-Hkt-point} for $p=2$.

\begin{proof}[Proof of Theorem \ref{thm-Hkt-point} where $p=2$]
	It follows by Lemma \ref{lem-all-p=2} and Theorem \ref{reduction-thm}.
\end{proof}

\section{Remarks on Maps between the Curves $H_{k,t}$}

The Hermitian curve $H_k$ is well known to be a maximal curve over $\mathbb F_{p^{2k}}$. We also proved this fact in Lemma \ref{lemma-Hk-max-min-odd}.

It is important to find the covering relationships between our curves $H_{k,t}$ and the well known Hermitian curves $H_k$. In the following propositions we will give the coverings of $H_{k,t}$ and show that each $H_{k,t}$ covers and is covered by some Hermitian curves.

\begin{prop}
Let $k$ be an integer and $t$ be an odd integer.	Then we have the following map $$\begin{matrix} H_{k,t} &\to &H_{k}\\
(x,y) &\mapsto &(x^{(p^{kt}+1)/(p^k+1)},y).
\end{matrix}$$
\end{prop}

\begin{prop}
	Let $k$ be an integer and $t$ be an odd integer. Then we have  the following map $$\begin{matrix} H_{kt} &\to &H_{k,t}\\ 
	(x,y) &\mapsto &(x,\sum\limits_{i=0}^{t-1}(-1)^iy^{p^{ki}}).
	\end{matrix}$$
\end{prop}

\begin{prop}
 Let $p=2$ and let $k$ and $t$ be integers. Then we have the following map $$\begin{matrix} H_{kt} &\to &H_{k,t}\\ 
	(x,y) &\mapsto &(x,\sum\limits_{i=0}^{t-1}y^{p^{ki}}).
	\end{matrix}$$
\end{prop}
The proofs are simple verification.

\section{Remarks on Maximal Curves}
Assume $t$ is odd if $p$ is odd.  Since $H_{k,t}$ is a maximal curve over $\mathbb F_{q^{2kt}}$, any curve covered by $H_{k,t}$ is a  maximal curve over $\mathbb F_{q^{2kt}}$. This follows easily from Theorem \ref{chapman:KleimanSerre}. This fact allows us to state the following theorems.

\begin{prop}
Let $k$ and $t$ be positive integers and assume $t$ is odd if $p$ is odd. Let $m\ge 2$ be a positive integer dividing  $p^{kt}+1$. Then the curve $$C: y^{p^k}+y=x^m$$ is a maximal curve over $\mathbb F_{q^{2kt}}$.
\end{prop}
If we apply the method in the proof of Lemma \ref{lemma-Hk-max-min-odd}, we have the following result.

\begin{prop}
Let $k$ and $t$ be positive integers and assume $t$ is odd if $p$ is odd.  Let $m\ge 2$ be a positive integer dividing  $p^{kt}+1$. Let $\mu$ be a nonzero root of $x^q+x=0$. Then the curve $$C: y^{p^k}-y=\mu x^m$$ is a maximal curve over $\mathbb F_{q^{2kt}}$.
\end{prop}

\section{Divisibility Property of the Curves $H_k$}
In this section, we will interest in $L$-polynomials of Hermitian curves and prove the Theorems \ref{Hkdiv} and \ref{Hknotdiv}. Similar proofs can be given for the curves $H_{k,t}$.
\begin{lemma}\label{divisibility-lemma-un} Let $k$ be a positive integer. 
	For $n\geq 1$ define 
	$$U_n=-p^{-n/2}[\#H_{2k}(\mathbb F_{p^n})-\#H_{k}(\mathbb F_{p^n})]$$and 
	write $U_n$ as a linear combination  
	of the $8k$-th roots of unity as
	$$U_n=\sum_{j=0}^{8k-1}u_jw_{s}^{-jn}.$$ 
	Then we have $$u_n \ge 0$$ for all $n \in \{ 0,1,\cdots, 8k-1\}$. 
\end{lemma}

\begin{proof}
	Let $U_0=U_s$.
	Write $k=2^vz$ where $v$ is a positive integer and $z$ is an odd integer.  
	By Corollary \ref{cor-Hk-point} we have $$U_n=\begin{cases}
	0 &\text{if } 2^{v+1} \nmid n, \\
	p^{(n,k)}(p^{(n,k)}-1) &\text{if } 2^{v+1} \mid \mid n,\\
	0 &\text{if } 2^{v+2} \mid \mid n,\\
		p^{(n,2k)}(p^{(n,2k)}-1)-	p^{(n,k)}(p^{(n,k)}-1) &\text{if } 2^{v+3} \mid n.
	\end{cases}$$ 
Therefore, by Inverse Fourier Transform we have that	
	\begin{align*}
	u_{n}&=\frac{1}{8k}\sum_{j=0}^{8k-1}U_jw_{8k}^{jn}\\
	&=\frac{1}{8k}\left[\sum_{j=0}^{2z-1}U_{2^{v+1}(2j+1)}w_{4z}^{jn} +\sum_{j=0}^{z-1}U_{2^{v+3}j}w_{2z}^{jn} \right] \quad \textrm{because
		$U_n=0$ if $2^{v+1} \nmid n$ or  $2^{v+2} \mid \mid n$}\\
	&\ge \frac{1}{8k}\left(U_0- \sum_{j=1}^{x-1}|U_{2^{v+3}j}| -\sum_{j=1}^{2z-1}|U_{2^{v+1}(2j+1)}| \right) \quad \textrm{by the
		triangle inequality}\\
	&\ge \frac{1}{8k}\left[p^{2k}(p^{2k}-1)-3zp^k(p^k-1)\right]  \quad \textrm{because
		$U_s=U_0=p^{2k}(p^{2k}-1)-p^k(p^k-1)$}\\
	& \qquad \qquad \qquad \qquad \qquad \qquad\qquad \qquad \qquad \qquad \textrm{    and all others are $\le p^k(p^k-1)$} \\
	&= \frac{1}{8k}p^k(p^k-1)(p^k(p^k+1)-3z)\\&\ge 0. 
	\end{align*}
	This finishes the proof.
\end{proof}

We write $L(H_k)$ for $L_{H_k}$.	

\begin{cor}\label{divisibility-2k}
	Let $k$ be a positive integer. Then $$L(H_k) \mid L(H_{2k}).$$
\end{cor}
\begin{proof}
	Lemma \ref{divisibility-lemma-un} shows that the multiplicity of each root of $L(H_k)$ is smaller than or equal to 
	its multiplicity as a root of $L(H_{2k})$.
\end{proof}

We do not know if there is a map from $H_{2k}$ to $H_k$.

\begin{lemma}\label{divisiblity-odd}
	Let $k$ be an integer and $t$ be an odd integer. Then $$L(H_k)\mid L(H_{kt}).$$
\end{lemma}

\begin{proof}
	Since $t$ is odd, $p^k+1$ divides $p^{kt}+1$ and therefore there is a map 
	of curves $H_{kt} \longrightarrow H_k$ given by
	$$(x,y)\to \left(x^{(p^{kt}+1)/(p^k+1)},\sum_{i=0}^{l-1}(-1)^iy^{p^{ik}}\right).$$ Hence 
	$L(X_k) \mid L(X_{kt})$ by Theorem \ref{chapman:KleimanSerre}.
\end{proof}

\begin{thm}\label{Hkdiv}
	Let $k$ and $m$ be positive integers. Then $$L(H_k) \mid L(H_{km}).$$
\end{thm}

\begin{proof}
	If $m=1$, then the result is trivial. Assume $m\ge 2$ and write $m=2^st$ where $t$ is odd. Since $t$ is odd, by Lemma \ref{divisiblity-odd} we have $$L(H_k)\mid L(H_{kt})$$ and by Corollary  \ref{divisibility-2k} we have $$L(H_{2^{i-1}kt})\mid L(H_{2^{i}kt})$$ for all $i\in\{1,\cdots,s\}$. Hence $$L(H_k) \mid L(H_{km}).$$
\end{proof}

\begin{thm}\label{Hknotdiv}
	Let $k$ and $\ell $ be positive integers such that $k$ does not divide $\ell$. 
	Then $L(H_k)$ does not divide $L(H_{\ell })$.
\end{thm}

\begin{proof}
	The period of $H_k$ is $4k$ and the period of $H_\ell$ is $4\ell$. Since $4k$ does not divide $4\ell$, we have $L(H_k)$ does not divide $L(H_\ell)$ by Proposition \ref{divisiblity-period}. 
\end{proof}

\section*{Acknowledgements}
We thank Gary McGuire for helpful conversations.

\end{document}